\documentclass[a4paper,10pt]{amsart}
\usepackage{amsmath,amssymb}
\usepackage{graphicx}
\usepackage{amsthm}
\usepackage{enumerate}
\theoremstyle{definition}
\newtheorem{theorem}{Theorem}[section]
\newtheorem{definition}[theorem]{Definition}
\newtheorem{lemma}[theorem]{Lemma}
 
\newtheorem{remark}[theorem]{Remark}
\newtheorem{notation}[theorem]{Notation}
\newtheorem{algorithm}[theorem]{Algorithm}
\newtheorem{observation}[theorem]{Observation}
\newtheorem*{acknowledgements}{Acknowledgements}
\title{Avoiding $\sigma$-porous sets in Hilbert Spaces.}

\begin{document}
\author{Michael Dymond}\footnote{The author was supported by EPSRC funding.}
\email{dymondm@maths.bham.ac.uk}
\address{School of Mathematics\\ University of Birmingham\\
Birmingham\\
B15 2TT\\
UK}
\begin{abstract}
We give a constructive proof that any $\sigma$-porous subset of a Hilbert space has Lebesgue measure zero on typical $C^{1}$ curves. Further, we discover that this result does not extend to all forms of porosity; we find that even power-$p$ porous sets may meet many $C^{1}$ curves in positive measure. 
\end{abstract}
%\begin{keyword}
%$\sigma$-porous sets \sep $\Gamma_{n}$-null sets
%\end{keyword}

\maketitle
\section{Introduction}
Porous and $\sigma$-porous sets form a class of exceptional sets which arise naturally in the study of differentiability of convex and Lipschitz functions. The concept of $\sigma$-porous sets was introduced by Dolzhenko in \cite{dolvzenko67}, and the connection between porosity and differentiability is investigated in \cite{lindenstrausstiserpreiss12}, \cite{preiss90} and \cite{preisszajicek01}. We introduce, in the following definition, various types of porosity which we will work with in this paper.
\begin{definition}\label{porous}
Let $(M,d)$ be a metric space, $c\in(0,1)$ and $p>1$.
\begin{enumerate}[(i)]
\item A subset $E$ of $M$ is said said to be $c$-porous at a point $x\in{E}$ if for every $\epsilon>0$ there exists a point $h\in{M}$ and a real number $r>0$ such that
\begin{equation*}
d(x,h)<\epsilon\mbox{, }B(h,r)\cap{E}=\emptyset\mbox{ and }r>c\cdot{d(x,h)}.
\end{equation*}
Here $B(h,r)$ denotes the open ball in $M$ with centre $h$ and radius $r$.
\item We say that $E$ is a $c$-porous subset of $M$ if $E$ is $c$-porous at every point $x\in{E}$. A subset $E$ of $M$ is called porous if it is $c$-porous for some $c\in(0,1)$.
\item A subset $P$ of $M$ is said to be $\sigma$-porous if it can be expressed as a countable union of porous subsets of $M$.
\item A subset $P$ of $M$ is said to be power-$p$-porous at a point $x\in{P}$ if whenever $\epsilon>0$, there exists $h\in{M}$ and $r>0$ such that $d(h,x)<\epsilon$, $B(h,r)\cap{P}=\emptyset$ and $r>d(h,x)^{p}$. $P$ is called power-$p$-porous if $P$ is power-$p$-porous at every point $x\in{P}$.
\end{enumerate}
\end{definition}
An immediate consequence of the definition is that $\sigma$-porous sets are of first category. Moreover, using the Lebesgue density theorem, $\sigma$-porous sets in finite dimensional spaces have Lebesgue measure zero. However, in infinite dimensional settings, where there is no analogue of the Lebesgue measure (see \cite{hunt92}), it is more difficult to describe the size of $\sigma$-porous sets. In our main result we exhibit a phenomenon of the $\sigma$-porous subsets of Hilbert spaces which reveals that they are also, in some sense, very small. Namely that any $\sigma$-porous subset of a Hilbert space is avoided by typical curves. More information about porosity and $\sigma$-porosity can be found in the survey \cite{zajicek75}, and, for a broader introduction to negligible sets relating to differentiability, see \cite[Chapter 6]{benyaminilindenstrauss00}. 

The notion of a set being $\Gamma_{n}$ null was introduced by Lindenstrauss, Preiss and Ti\v{s}er in \cite[Chapter 5]{lindenstrausstiserpreiss12}. For a Banach space $X$, we denote by $\Gamma_{n}(X)$ the Banach space $C^{1}([0,1]^{n},X)$ of continuously differentiable maps from $[0,1]^{n}$ to $X$.
\begin{definition}\label{gammanull}
A Borel subset $A$ of a Banach space $X$ is called $\Gamma_{n}$-null if 
\begin{equation*}
\mathcal{L}^{n}\left\{t\in[0,1]^{n}:\gamma(t)\in{A}\right\}=0
\end{equation*}
for residually many $\gamma\in\Gamma_{n}(X)$.
\end{definition}
It has been established, in the recent book \cite{lindenstrausstiserpreiss12} of Lindenstrauss, Preiss and Ti\v{s}er, that every $\sigma$-porous subset of a separable Banach space having a separable dual is $\Gamma_{1}$-null \cite[Theorem 10.4.1]{lindenstrausstiserpreiss12}. In the present paper we give a new proof for $\sigma$-porous sets in Hilbert spaces. Our argument is noteworthy because it is constructive and presents a method of finding many curves which avoid a given porous set. 

There have been notable discoveries of the opposite nature. In, \cite{speight12} Speight proves that for every Banach space $X$ and integer $n$ with $2<n<\dim X$, there exists a directionally porous set $P\subset X$ which is not $\Gamma_{n}$-null. Thus, our main result fails if we try to replace curves with $n$-dimensional surfaces, where $n>2$. Further, \cite{maleva07} verifies the existence of an `unavoidable' $\sigma$-porous set - a $\sigma$-porous set whose complement is null on all Lipschitz curves. Such a set can be found in any Banach space containing $l_{1}$.
\subsection{Structure of the paper}
We work in a fixed Hilbert space $H$. In Section 2, we introduce a method of altering a given curve in $\Gamma_{1}(H)$ in order to obtain a nearby curve which avoids a given porous set $E$. This is achieved by pushing segments of the original curve inside the holes of $E$. Critical to our approach will be an application of the Vitali Covering Theorem \cite[Chapter 2]{mattila99} which allows us to nominate suitable segments of the curve to modify.   

The crucial lemma (Lemma \ref{mainlemma}) in the proof of our main result is stated at the beginning of Section 3 and the majority of this section is devoted to its proof. We present an algorithm in which we repeatedly apply the results established in Section 2, in order to construct a finite sequence of curves so that the final curve is close to the original and we have good control over the measure of a porous set $E$ on all curves in some neighbourhood. The biggest problem we face is ensuring that the difference between the final curve and our original does not become too large. We overcome this difficulty by ensuring that the difference in the derivatives behaves like a martingale. This allows us to control the size of the set where the difference in derivatives becomes too large. Here, we require Kolmogorov's Martingale Theorem, quoted below from \cite[p237]{taylor97}:
\begin{theorem}\label{kolmogorov}
Let $\left\{M_{n}\right\}$ be a martingale with respect to the filtration
$\left\{\mathcal{G}_{n}\right\}$ on a probability space
$(\Omega,\mathcal{G},\mathbb{P})$. Assume
$\mathbb{E}\left(\left\|M_{n}\right\|^{2}\right)<\infty$ for all $n$. Then
for any $\kappa>0$
\begin{equation*}
\mathbb{P}\left[\max_{1\leq{n}\leq{N}}\left\|M_{n}\right\|\geq{\kappa}\right]\leq\frac{\mathbb{E}\left(\left\|M_{N}\right\|^{2}\right)}{\kappa^{2}}.
\end{equation*}
\end{theorem}
Finally, in Section 4 we investigate the behaviour of power-$p$-porous sets on $C^{1}$-curves, finding that they can behave rather differently, in this respect, to conventional porous or $\sigma$-porous sets. We prove that there is a measure zero power-$p$-porous subset of the plane, which is not $\Gamma_{1}$-null.
\begin{notation}\label{notation}
For the fixed Hilbert space $H$ and a mapping $f:[0,1]\to{H}$, we shall write $\left\|f\right\|$ for the sup-norm of $f$. In particular the norm of a curve $\gamma\in\Gamma_{1}(H)$ is expressed as $\left\|\gamma\right\|+\left\|\gamma'\right\|$. We will also use $\left\|-\right\|$ to denote the norm on the Hilbert space $H$. It will be clear from the context which norm is intended. Given a mapping $\phi:[0,1]\to H$ and a point $t\in[0,1]$ where $\phi$ is differentiable, we will often identify the linear map $\phi'(t)\in\mathcal{L}([0,1],H)$ with the unique vector $v\in H$ satisfying $\phi'(t)(s)=vs$ for all $s\in[0,1]$. For a subset $S$ of $[0,1]$ we denote by $\left|S\right|$, the outer Lebesgue measure of $S$. %Given a real number $A$, we write $[A]$ for the integer part of $A$ . 
For a set $D\subseteq H$ we denote the interior of $D$ by $\mbox{Int}(D)$ and the closure of $D$ by $\mbox{Clos}(D)$.
\end{notation}
\begin{acknowledgements}
The author wishes to thank Professor David Preiss for a huge amount of guidance and support. The author also wishes to thank Dr. Olga Maleva for helpful discussions.
\end{acknowledgements}
\section{Avoiding a Porous Set}
Let $H$ be a Hilbert space and fix a $c$-porous set $E\subseteq H$. In this section we will describe how to modify a $C^{1}$ curve in $H$ so that the new curve avoids $E$ on some open subset of $[0,1]$. Moreover, we obtain a lower bound for the measure of this open set, using the porosity of $E$. To begin, we introduce a collection of candidate intervals on which one might consider altering the starting curve. 
\begin{lemma}\label{vitalicover} Suppose $\lambda>1$ and $\theta>0$
are real numbers. Let $f:[0,1]\to{H}$ be a $C^{1}$ curve. Then
there exists a collection $\mathcal{V}=\mathcal{V}(f,\lambda,\theta)$ of
closed subintervals of $(0,1)$ with the following properties.
\begin{enumerate}
\item $\mathcal{V}$ is a Vitali cover of $f^{-1}(E)\cap (0,1)$.
\item Each interval $I$ in $\mathcal{V}$ is of the form $I=\left[x-\lambda{d},x+\lambda{d}\right]\subset(0,1)$, where $x\in{f^{-1}(E)}$ and $d>0$ are such that there exists $h\in{H}$ and $r>0$ satisfying $d=\left\|h-f(x)\right\|<\theta$, $B(h,r)\cap{E}=\emptyset$ and $r>cd$.
\end{enumerate}
\end{lemma}
\begin{proof} Given $x\in{f^{-1}(E)\cap(0,1)}$ and $\xi\in(0,\theta)$ we can find,
using the porosity of $E$ at $f(x)$, a point $h\in{H}$ with
\begin{equation*}\label{xi}
d=\left\|h-f(x)\right\|<\min\left\{\xi,x/\lambda,(1-x)/\lambda\right\}
\end{equation*}
and a real number $r>0$ such that the conditions of property 2 are satisfied. Define a closed interval $I(x,\xi)=[x-\lambda d,x+\lambda d]\subset(0,1)$. %We note that $\left|I(x,\xi)\right|=2\lambda{d}\leq{2\lambda\xi}$. 
The desired collection $\mathcal{V}$ is now defined by 
\begin{equation*}
\mathcal{V}=\left\{I(x,\xi)\mbox{ : }x\in{f^{-1}(E)\cap(0,1)}\mbox{ and
}0<\xi<\theta\right\}.
\end{equation*}
\end{proof}
By applying the Vitali Covering Theorem, we can extract a pairwise
disjoint collection of intervals from $\mathcal{V}$, which efficiently covers almost all of $f^{-1}(E)$. Our strategy is then to slightly modify the curve $f$ on these intervals so that the resulting curve passes through ``holes'' of $E$. In what follows, we construct a piecewise linear curve $\psi$, with the property that $f+\psi$ avoids the porous set $E$ on some open subsets of $[0,1]$. 

Suppose $M,\lambda>1$ and $\theta>0$ are real numbers. Let $f:[0,1]\to{H}$ be a $C^{1}$ curve with $\left\|f'\right\|\leq{M}$. For the remainder of this section, we let $\mathcal{I}=\left\{I_{k}\right\}_{k=1}^{\infty}$ denote a sequence of pairwise disjoint intervals $I_{k}$ from $\mathcal{V}\left(f,\lambda,\theta\right)$.
\begin{definition}\label{piecewisedef}
We define a collection of piecewise linear, continuous curves
\begin{equation}\label{phi}
\left\{\psi_{K}=\psi(f,\lambda,\theta,M,\mathcal{I},K):[0,1]\to{H}\mbox{ : }K\in\mathbb{N}\right\}
\end{equation}
by the following discussion. Let $K\geq{1}$ be an integer. We set
\begin{equation}\label{zero}
\psi_{K}(t)=0\mbox{ whenever }t\notin{\bigcup_{1\leq{k}\leq{K}}I_{k}}.
\end{equation}
Next, suppose $I$ is one of the intervals $I_{1},\ldots,I_{K}$. By
Lemma \ref{vitalicover}, there exists $x\in{f^{-1}(E)}$, $h\in H$, $0<d=\left\|h-f(x)\right\|<\theta$ and $r>0$ such that the conditions of Lemma \ref{vitalicover}, part 2 are satisfied. Let $\varphi:I\to[0,1]$ be the piecewise linear, continuous function
defined by
\begin{equation}\label{varphi1}
\varphi(t)=\begin{cases} 1-\frac{x-t}{\lambda{d}} & \mbox{ if
}t\in\left[x-\lambda{d},x\right] \\ 1-\frac{t-x}{\lambda{d}} & \mbox{ if }t\in\left[x,x+\lambda{d}\right],
\end{cases}
\end{equation}
and set
\begin{equation}\label{varphi}
\psi_{K}(t)=\varphi(t)(h-f(x))\mbox{ for all }t\in{I}.
\end{equation} 
\end{definition}
\begin{observation}\label{piecewise}
The following properties of the collection $\left\{\psi_{K}:K\in\mathbb{N}\right\}$ are immediate from Definition~\ref{piecewisedef} and Lemma~\ref{vitalicover}.
\begin{enumerate}
\item For $1\leq{k}\leq{K}$, each interval $I_{k}=[a_{k},b_{k}]$ admits a vector $p_{k}\in{H}$, a real number $\alpha_{k}$ and a
scalar function $\varphi_{k}:I_{k}\to[0,1]$ such that %with $\displaystyle\int_{I_{k}}\varphi_{k}'=0$, such that
\begin{equation}\label{varphidef}
\psi_{K}(t)=\varphi_{k}(t)p_{k}\mbox{ for all }t\in{I_{k}},
\end{equation}
and 
\begin{equation*}
\varphi_{k}'(t)=\begin{cases}
\alpha_{k} & \mbox{ if }t\in(a_{k},x_{k})\\
-\alpha_{k} & \mbox{ if }t\in(x_{k},b_{k})
\end{cases}
\end{equation*}
where $x_{k}=(a_{k}+b_{k})/2$ is the midpoint of the interval $I_{k}$. In particular we have $\int_{I_{k}}\phi_{k}'=0$ for each $k$. We further note that the points where the mapping $\psi_{K}$ is not differentiable are precisely the endpoints $a_{k},b_{k}$ and the midpoint $x_{k}$ of each interval $I_{k}$ with $1\leq k\leq K$.
\item $\left\|\psi_{K}\right\|<\theta$, and $\left\|\psi_{K}'(t)\right\|\leq\frac{1}{\lambda}$ whenever $\psi_{K}$ is differentiable at $t$.
\item If $L\leq{K}$ then $\psi_{L}(t)=\begin{cases} \psi_{K}(t) & \mbox{ if }t\in{I_{k}}\mbox{ and }1\leq{k}\leq{L} \\
0 & \mbox{otherwise}
\end{cases}$.
\end{enumerate}
\end{observation}
%\begin{remark}\label{identification}
%For a piecewise-linear mapping $\psi:[0,1]\to H$, and a point $t\in[0,1]$ where $\psi$ is differentiable, we will identify the linear map $\psi'(t)\in\mathcal{L}([0,1],H)$ with the unique vector $v\in H$ satisfying $\psi'(t)(s)=vs$ for all $s\in[0,1]$.
%\end{remark}
\begin{lemma}\label{holeestimate}
Let $K\geq{1}$ be an integer and set $\psi=\psi_{K}$. For $1\leq{k}\leq{K}$, each interval $I_{k}$ has an open subinterval $R_{k}$ such
that the following conditions hold:
\begin{enumerate}
\item There exists an open ball $B_{k}$ in $H$ with
$(f+\psi)(R_{k})\subset{B_{k}}\subset{H\setminus{E}}$.
\item $\left|R_{k}\right|\geq\frac{Q}{\lambda}\left|I_{k}\right|$
where $Q=c/4M$.
\item The midpoint $x_{k}$ of $I_{k}$ is contained in $R_{k}$ and $(f+\psi)(x_{k})$ is the centre of the ball $B_{k}$.
\item The closure of $R_{k}$ does not contain the endpoints of $I_{k}$.
\end{enumerate}
\end{lemma}
\begin{proof} Let $I=I_{k}$ be one of the intervals $I_{1},\ldots,I_{K}$. Let the point $x\in[0,1]$, point $h\in{H}$ and a real number $d>0$ be given by the discussion in Definition~\ref{piecewisedef}. Put $B=B_{k}=B(h,r)$. From \eqref{varphi} and \eqref{varphi1} we have that
$(f+\psi)(x)=h$. Therefore, the set $U=(f+\psi)^{-1}(B)$ is a non-empty, (open) subset of $[0,1]$, containing the point $x$. Define an open subinterval $R=R_{k}$ of $I=I_{k}$ by
\begin{equation}\label{R}
R=U(x)\cap\left(x-\lambda d/2,x+\lambda d/2\right)\mbox{,}
\end{equation}
where $U(x)$ denotes the connected component of $U$ containing $x$. Parts 1, 3 and 4 of Lemma \ref{holeestimate} are now readily verified. To prove part 2, we distinguish between two cases. If $R=\left(x-\frac{\lambda{d}}{2},x+\frac{\lambda{d}}{2}\right)$ then using Lemma \ref{vitalicover}, part 2, $Q=c/4M$ and $\lambda>1$ we have
\begin{equation*}
\left|R\right|=\lambda{d}=\frac{1}{2}\left|I\right|\geq\frac{Q}{\lambda}\left|I\right|.
\end{equation*}
We can therefore assume that some point in the interval $\left(x-\frac{\lambda{d}}{2},x+\frac{\lambda{d}}{2}\right)$ is mapped
outside of $B(h,r)$ by $(f+\psi)$. In this case, there exists a point
$y\in\left(x-\lambda d/2,x+\lambda d/2\right)$ with $(f+\psi)(y)\in\partial{B(h,r)}$. Without loss of generality we assume $y>x$. Note that $\psi$ is differentiable on the open interval $(x,y)$. Hence,
\begin{align}
\int_{x}^{y}\left\|f'(t)+\psi'(t)\right\|dt&\geq\left\|(f+\psi)(y)-(f+\psi)(x)\right\|= \nonumber \\
&=\left\|(f+\psi)(y)-h\right\|=r>cd=\frac{c\left|I\right|}{2\lambda}. \label{integralbound1}
\end{align}
The interval $(x,y)$ is a subset of $R$. Consequently,
$\left|y-x\right|\leq\left|R\right|$. Using this fact, along with $\left\|f'\right\|\leq M$ and part 2 of Observation \ref{piecewise} we obtain
\begin{equation}
\int_{x}^{y}\left\|f'(t)+\psi'(t)\right\|dt\leq\int_{x}^{y}\left\|f'(t)\right\|dt+\int_{x}^{y}\left\|\psi'(t)\right\|dt\leq {M\left|R\right|+\frac{1}{\lambda}\left|R\right|}<2M\left|R\right|. \label{integralbound2}
\end{equation}
Together, \eqref{integralbound1}, \eqref{integralbound2} and $Q=c/4M$ imply that $\left|R\right|\geq\frac{Q}{\lambda}\left|I\right|$, as required.
\end{proof}
\begin{lemma}\label{smoothing}
Let $\psi=\psi(f,\lambda,\theta,M,\mathcal{I},K)$  and let the open intervals $R_{1},\ldots,R_{K}$ be given by the conclusion of Lemma \ref{holeestimate}. Suppose $\epsilon>0$. Then there exists a function $g=g(f,\psi,\epsilon):[0,1]\to{H}$ such that the following statements hold.
\begin{enumerate}
\item $g$ is continuously differentiable on $[0,1]$.
\item $\left\|g-f\right\|<\theta$, and $\left\|g'-f'\right\|\leq 2\sup\left\{\left\|\psi'(t)\right\|:\psi\mbox{ is differentiable at }t\right\}$.
\item There exists a closed set $T\subseteq\bigcup_{k=1}^{K}I_{k}$ such that $T$ is a finite union of closed intervals, $\left|T\right|<\epsilon$ and $t\in T$ whenever $g(t)\neq f(t)+\psi(t)$ or $\psi$ is not differentiable at $t$.
\item $g(t)=f(t)$ whenever $t\notin\displaystyle\bigcup_{k=1}^{K}I_{k}$.
\item For each $k$, there exists an open ball $B_{k}$ in $H$ with
$g(R_{k})\subset{B_{k}}\subset{H\setminus{E}}$.
\end{enumerate}
\end{lemma}
\begin{proof}
Write $I_{k}=[a_{k},b_{k}]$ for $k=1,\ldots,K$. For each $k$ we let $x_{k}$ denote the midpoint of the interval $I_{k}$. %unique point in the open interval $R_{k}$ where $\psi$ is not differentiable. 
From Lemma~\ref{holeestimate}, part 3 we have that $x_{k}\in R_{k}$ and $(f+\psi)(x_{k})$ is the centre of an open ball $B_{k}\subset{H}$ with $(f+\psi)(R_{k})\subset{B_{k}}\subset{H\setminus{E}}$. For $1\leq{k}\leq{K}$ and $\delta>0$ we set
\begin{equation*}
J_{\delta}(a_{k})=[a_{k},a_{k}+\delta]\mbox{, }J_{\delta}(b_{k})=[b_{k}-\delta,b_{k}]\mbox{ and }J_{\delta}(x_{k})=[x_{k}-\delta,x_{k}+\delta].
\end{equation*}
Let $s(\psi)=\sup\left\{\left\|\psi'(t)\right\|\mbox{ : }\psi\mbox{ is differentiable at }t\right\}$.
Using Lemma \ref{holeestimate}, part 4 and Observation \ref{piecewise}, part 2 we may choose $\rho\in(0,\min\left\{\epsilon/12K,(\theta-\left\|\psi\right\|)/12Ks(\psi)\right\})$ small enough so that, for each $1\leq{k}\leq{K}$, the sets $J_{\rho}(a_{k}),J_{\rho}(b_{k}),R_{k}$ are pairwise disjoint subsets of the interval $I_{k}$. For each $x_{k}$, we may choose $r_{k}>0$ such that
\begin{equation*}
B_{k}=B((f+\psi)(x_{k}),r_{k})\subset{H\setminus{E}}.
\end{equation*}
Choose $0<\zeta_{k}<\min\left\{\rho,r_{k}/12s(\psi)\right\}$ sufficiently small so that $J_{\zeta_{k}}(x_{k})\subset{R_{k}}$ and $\left\|(f+\psi)(t)-(f+\psi)(x_{k})\right\|<r_{k}/2$ for all
$t\in{J_{\zeta_{k}}(x_{k})}$. Set  
\begin{equation*}
T=\bigcup_{k=1}^{K}\left(J_{\rho}(a_{k})\cup{J_{\rho}(b_{k})}\cup{J_{\zeta_{k}}(x_{k})}\right),
\end{equation*} 
and observe that $\left|T\right|=\rho{K}+\rho{K}+\sum_{k=1}^{K}2\zeta_{k}<\epsilon$. From Observation~\ref{piecewise}, part 1 we have that $\psi$ is continuously differentiable
on the set $[0,1]\setminus{T}$. Let $\xi:[0,1]\to{H}$ be a
continuous map with 
\begin{multline}\label{eq:xi}
\xi(t)=\psi'(t)\mbox{ for all }t\in[0,1]\setminus{T}\mbox{, }\left\|\xi\right\|\leq 2s(\psi)\mbox{ and }\\
\int_{J}\xi=\int_{J}\psi'\mbox{ whenever }J=J_{\rho}(a_{k}),J_{\rho}(b_{k}),J_{\zeta_{k}}(x_{k}). 
\end{multline}
We may define $\xi$ as follows: Set $\xi(t)=\psi'(t)$ for all $t\in[0,1]\setminus T$. Note that $\psi'$ restricted to any one of the intervals $(a_{k},a_{k}+\rho)$, $(b_{k}-\rho,b_{k})$, $[x_{k}-\zeta_{k},x_{k})$ and $(x_{k},x_{k}+\zeta_{k}]$ is constant. This follows from Observation~\ref{piecewise}, part 1 and the condition that the intervals $J_{\rho}(a_{k})$, $J_{\rho}(b_{k})$, $J_{\zeta_{k}}(x_{k})$ are pairwise disjoint subsets of $I_{k}$ for each $k$. Further, by \eqref{zero} we have $\psi'(t)=0$ whenever $t\notin \bigcup_{1\leq k\leq K}I_{k}$. On each interval $J_{\rho}(a_{k})$, define $\xi$ by stipulating that $\xi$ is affine on the intervals $[a_{k},a_{k}+\rho/2]$, $[a_{k}+\rho/2,a_{k}+\rho]$ and
\begin{equation*}
\xi(t)=\begin{cases} 
0=\psi'(t) & \mbox{ if }t=a_{k},\\
3\psi'(t)/2 & \mbox{ if }t=a_{k}+\rho/2,\\
\psi'(t) & \mbox{ if }t=a_{k}+\rho.
\end{cases}
\end{equation*}
$\xi$ may be defined on the intervals $J_{\rho}(b_{k})$ similarly. Next, using Observation~\ref{piecewise}, part 1, note that $\psi'(x_{k}-\zeta_{k})=-\psi'(x_{k}+\zeta_{k})$ and $\int_{J_{\zeta_{k}}(x_{k})}\psi'=0$. Thus, we may define $\xi$ on each interval $J_{\zeta_{k}}(x_{k})$ by stipulating that $\xi$ is affine on the interval $J_{\zeta_{k}}(x_{k})$ and $\xi(t)=\psi'(t)$ for $t=x_{k}\pm\zeta_{k}$.

We now define a curve $\eta$ by $\eta(t)=\int_{0}^{t}\xi(u)du\mbox{ for all }t\in[0,1]$.

Finally, we set $g=f+\eta$. Part 1 of the present lemma is now clear. Using \eqref{eq:xi}, we get that $\eta(t)=\int_{0}^{t}\psi'(u)du=\psi(t)-\psi(0)=\psi(t)$ whenever $t\in[0,1]\setminus T$; part 3 is now readily verified. Part 4 is an immediate consequence of part 3 and \eqref{zero}. Let us now verify part 2: We fix $t\in[0,1]$ and  observe that %aim to show that $\left\|g(t)-f(t)\right\|<\theta$. %For future use, we note that $\psi(0)=\eta(0)=0$. This follows from the conditions $I_{k}\subset(0,1)$, \eqref{zero} and statement 4 of the present lemma. We can therefore write $\eta(t)-\psi(t)=\int_{0}^{t}\xi(u)du-\int_{0}^{t}\psi'(u)du$. We now observe that
\begin{align*}
\left\|g(t)-(f+\psi)(t)\right\|&=\left\|\eta(t)-\psi(t)\right\|\\
&=\left\|\int_{0}^{t}\xi(u)du-\int_{0}^{t}\psi'(u)du\right\|\\
&\leq\int_{0}^{t}\left\|\xi(u)-\psi'(u)\right\|du\\
&\leq\int_{T}3s(\psi)du\leq \rho\cdot 4K\cdot 3s(\psi)<\theta-\left\|\psi\right\|
\end{align*}
In the above we use \eqref{eq:xi}, %part 4 of the present lemma 
and the restrictions on $\rho$ and $\zeta_{k}$ given above. The inequality $\left\|g-f\right\|<\theta$ now follows, whilst $\left\|g'-f'\right\|\leq 2s(\psi)$ is readily verified using the definitions of $g$, $\eta$ and $\xi$.

Finally, we turn our attention to part 5. If $t\in R_{k}\setminus T$ then, by part 3 and Lemma~\ref{holeestimate}, part 1, we have that $g(t)=(f+\psi)(t)\in B_{k}$. Since the intervals $J_{\rho}(a_{k})$ and $J_{\rho}(b_{k})$ were chosen to be disjoint from $R_{k}$, it only remains to show that points $t\in J_{\zeta_{k}}(x_{k})$ satisfy $g(t)\in B_{k}$. Fix $1\leq k\leq K$ and $t\in J_{\zeta_{k}}(x_{k})$. We observe that
\begin{align*}
\left\|g(t)-(f+\psi)(x_{k})\right\|&\leq\left\|(f+\eta)(t)-(f+\psi)(t)\right\|+\left\|(f+\psi)(t)-(f+\psi)(x_{k})\right\|\\
&\leq\left\|\int_{0}^{t}\xi(u)du-\int_{0}^{t}\psi'(u)du\right\|+r_{k}/2\\
&\leq\int_{x_{k}-\zeta_{k}}^{t}\left\|\xi(u)-\psi'(u)\right\|du+r_{k}/2\\
&\leq 2\zeta_{k}\cdot 3s(\psi)+r_{k}/2<r_{k}.
\end{align*}
Hence $g(t)\in B((f+\psi)(x_{k}),r_{k})=B_{k}$. % To get the third inequality in the sequence above, we use \eqref{eq:xi} to prove that $\int_{0}^{x_{k}-\zeta_{k}}\xi(u)-\psi'(u)du=0$. 
\end{proof}
The next Lemma reveals that by careful choice of our starting curve $f$, we gain some control over the measure of $E$ on curves in a neighbourhood of the new curve $g(f,\psi,\epsilon)$.
\begin{lemma}\label{banachmazur}
Suppose $M,\lambda>1$ and $\theta,\epsilon,\upsilon>0$ are real
numbers. Let $U$ be an open subset of $\Gamma_{1}(H)$ and $f\in{U}$ be a curve
satisfying $\left\|f'(t)\right\|\leq{M}$ for all $t\in[0,1]$ and
\begin{equation}\label{closetoworst}
\left|f^{-1}(E)\right|>\sup_{\gamma\in{U}}\left|\gamma^{-1}(E)\right|-\upsilon.
\end{equation}
Suppose $F$ is a finite union of closed intervals in $[0,1]$, let $\mathcal{I}=\left\{I_{k}\right\}_{k=1}^{\infty}$ be a collection of pairwise disjoint intervals from $\mathcal{V}\left(f,\lambda,\theta\right)$ such that $I_{k}\subseteq[0,1]\setminus F$ for each $k$ and let $L\geq 1$ satisfy
\begin{equation}\label{int2}
\left|\left(f^{-1}(E)\setminus{F}\right)\setminus\left(\bigcup_{k=1}^{\infty}I_{k}\right)\right|=0\mbox{ and }\sum_{k=L+1}^{\infty}\left|I_{k}\right|<\upsilon.
\end{equation}
Let the function $\psi=\psi(f,\lambda,\theta,M,\mathcal{I},L)$ and the open
intervals $R_{1},\ldots,R_{L}$ be given by Definition \ref{piecewisedef} and Lemma \ref{holeestimate} respectively. Let the curve
$g=g(f,\psi,\epsilon)$ be given by the conclusion of Lemma \ref{smoothing}.
Define subsets $S$ and $A$ of $[0,1]$ by
\begin{equation*}
S(g,f)=\bigcup_{k=1}^{L}R_{k}\mbox{ and }
A(g,f,F)=[0,1]\setminus\left(\bigcup_{k=1}^{L}I_{k}\cup F\right).
\end{equation*}
Then there exists $\delta>0$ such that for all curves $\gamma$ satisfying 
\begin{equation}\label{condition}
\sup_{t\in S\cup A}\left\|\gamma(t)-g(t)\right\|+\sup_{t\in S\cup A}\left\|\gamma'(t)-g'(t)\right\|\leq \delta,
\end{equation}
we have
\begin{equation*}
\left|\gamma^{-1}(E)\cap\mbox{Clos}(S)\right|<\upsilon\mbox{ and }\left|\gamma^{-1}(E)\cap\mbox{Clos}(A)\right|<5\upsilon.
\end{equation*}
\end{lemma}
\begin{proof}  %It is straightforward to check that the sets $\mbox{Clos}(S)\setminus{S}$ and $\mbox{Clos}(A)\setminus{A}$ have measure zero. Therefore, it is enough to show that there exists $\delta>0$ such that 1 and 2 hold with $\mbox{Clos}(D)$ replaced with $D$ for $D=S,A$.  
We focus only on finding a suitable $\delta$ so that the second condition above holds. Choosing $\delta$ small enough for the first condition is an easy application of Lemma~\ref{smoothing}, part 5. Note that $\mbox{Clos}(A)\setminus A$ has measure zero. It is therefore enough to verify the second condition with $\mbox{Clos}(A)$ replaced by $A$. 

%Suppose first that the set $F$ has positive measure. 
Using the Vitali Covering Theorem, we may extract a countable collection $\mathcal{I}'=\left\{I'_{k}\right\}_{k=1}^{\infty}$ of pairwise disjoint intervals from $\mathcal{V}(f,\lambda,\theta)$ and an integer $L'\geq 1$ such that $I_{l}'\cap I_{k}=\emptyset$ for all $l$ and $1\leq k \leq L$, and condition \eqref{int2} also holds when $(f^{-1}(E)\setminus F,I_{k},L)$ are replaced by $(f^{-1}(E)\cap F,I_{k}',L')$. %If $\left|F\right|=0$ then we set $L'=0$.
%\begin{equation}\label{int'2}
%\left|\left(f^{-1}(E)\cap{F}\right)\setminus\left(\bigcup_{k=1}^{\infty}I_{k}'\right)\right|=0
%\end{equation}
%We now choose $K'\geq{1}$ such that \eqref{throwout} is satisfied with $K=K'$.
%\begin{equation}\label{throwout'}
%\sum_{k=K'+1}^{\infty}\left|I_{k}'\right|<\upsilon.
%\end{equation}
For each $1\leq k\leq L$ and $1\leq l\leq L'$, write $I_{k}=[a_{k},b_{k}]$ and $I_{l}'=[a_{l}',b_{l}']$.
%\begin{equation*}
%I_{k}=[a_{k},b_{k}]\mbox{ where }0\leq{a_{k}}<b_{k}\leq{1}\mbox{ for each %}k=1,\ldots,K\mbox{ and}
%\end{equation*}
%\begin{equation*}
%I_{k}'=[a_{k}',b_{k}']\mbox{ where }0\leq{a_{k}'}<b_{k}'\leq{1}\mbox{ for each %}k=1,\ldots,K'.
%\end{equation*}
Next, choose $\omega\in(0,\upsilon/2(L+L'))$ small enough so that the intervals 
\begin{equation*}
\left\{[a_{k}-\omega,b_{k}+\omega],[a_{k}'-\omega,b_{k}'+\omega]\mbox{ : }1\leq{k}\leq{L},1\leq{k}\leq{L'}\right\}
\end{equation*}
are pairwise disjoint. Pick $\tau>0$ small enough so that $B(f,\tau)\subset{U}$ and choose $0<\delta<\omega\tau/(4\pi)$. Suppose for a contradiction that there exists a curve $\beta\in\Gamma_{1}$ such that \eqref{condition} holds with $\gamma=\beta$, and $\left|\beta^{-1}(E)\cap{A}\right|\geq 5\upsilon$. We may then define a cuve $\widetilde{\beta}:[0,1]\to{H}$ by
\begin{equation}\label{beta}
\widetilde{\beta}(t)=\begin{cases}
f(t) &\mbox{if }t\in{\displaystyle\bigcup_{k=1}^{L}{I_{k}}\cup\bigcup_{k=1}^{L'}I_{k}'}, \\
\frac{1}{2}[\cos\frac{\pi}{\omega}(t-a_{k}+\omega)+1]\left(\beta(t)-f(t)\right)+f(t)
&\mbox{if }t\in[a_{k}-\omega,a_{k}],1\leq{k}\leq{L}, \\
\frac{1}{2}[\cos\frac{\pi}{\omega}(t-b_{k})+1]\left(f(t)-\beta(t)\right)+\beta(t)
&\mbox{if }t\in[b_{k},b_{k}+\omega],1\leq{k}\leq{L}, \\
\frac{1}{2}[\cos\frac{\pi}{\omega}(t-a_{k}'+\omega)+1]\left(\beta(t)-f(t)\right)+f(t)
&\mbox{if }t\in[a_{k}'-\omega,a_{k}'],1\leq{k}\leq{L'}, \\
\frac{1}{2}[\cos\frac{\pi}{\omega}(t-b_{k}')+1]\left(f(t)-\beta(t)\right)+\beta(t)
&\mbox{if }t\in[b_{k}',b_{k}'+\omega],1\leq{k}\leq{L'}, \\
\beta(t) &\mbox{otherwise.}
\end{cases}
\end{equation}
It is readily verified, using Lemma \ref{smoothing} part 4 and the fact that $\gamma=\beta$ satisfies \eqref{condition}, that $\widetilde{\beta}$ is a $C^{1}$ curve with $\widetilde{\beta}\in{B(f,\tau)}\subseteq{U}$. We also note that $\widetilde{\beta}(t)=\beta(t)$ for all
$t\in{A\setminus\left(\displaystyle\bigcup_{k=1}^{L}[a_{k}-\omega,a_{k}]\cup[b_{k},b_{k}+\omega]\cup\bigcup_{k=1}^{L'}[a_{k}'-\omega,a_{k}]\cup[b_{k}',b_{k}'+\omega]\right)}$.
Combining this with \eqref{beta}, \eqref{int2} and \eqref{closetoworst} we deduce the following:
\begin{align*}
\left|\widetilde{\beta}^{-1}(E)\right|&\geq\sum_{k=1}^{L}\left|f^{-1}(E)\cap{I_{k}}\right|+\sum_{k=1}^{L'}\left|f^{-1}(E)\cap{I_{k}'}\right|+\left|\beta^{-1}(E)\cap{A}\right|-2(L+L')\omega \\
&>\left|f^{-1}(E)\setminus F\right|-\sum_{k=L+1}^{\infty}\left|I_{k}\right|+\left|f^{-1}(E)\cap F\right|-\sum_{k=L'+1}^{\infty}\left|I_{k}'\right|+5\upsilon-\upsilon\\
&>\left|f^{-1}(E)\right|-2\upsilon+5\upsilon-\upsilon=\left|f^{-1}(E)\right|+2\upsilon>\sup_{\gamma\in{U}}\left|\gamma^{-1}(E)\right|.
\end{align*}
This is incompatible with $\widetilde{\beta}\in U$.
\end{proof}
%In Lemma \ref{smoothing} it is established that the difference between the new curve $g(f,\psi,\epsilon)$ and the starting curve $f$ can be well approximated by the piecewise linear map $\psi$. Thus, after successive iterations of our method, the difference in derivative between the resulting curve and our original curve behaves like the sum of the derivatives of finitely many piecewise linear maps of the form given in Definition \ref{piecewisedef}. We now prove a Lemma which asserts that such sums have the special property of forming a martingale. This will be important in allowing us to control the difference in derivative between the final constructed curve and the original.
\section{Construction}
We begin this section by stating the crucial lemma in the proof of our main result. As in the previous section, the setting is a real Hilbert space $H$ in which we fix a $c$-porous set $E$. 
\begin{lemma}\label{mainlemma}
For any open ball $U$ in $\Gamma_{1}(H)$ there exists an open ball $V$ in $\Gamma_{1}(H)$ such that $V\subset{U}$ and $\sup_{f\in{V}}\left|f^{-1}(E)\right|<\frac{1}{2}\sup_{f\in{U}}\left|f^{-1}(E)\right|$.
\end{lemma}
The majority of this section will be devoted to proving Lemma \ref{mainlemma}. Fix
\begin{equation}\label{eq:epsilon}
0<\epsilon<\frac{1}{72}\sup_{f\in{U}}\left|f^{-1}(E)\right|.
\end{equation}
We begin the construction by setting $U_{0}=U$, $F_{0}=C_{0}=\emptyset$, $\delta_{0}=1$ and choosing a curve $f_{1}\in U$ satisfying
\begin{equation}\label{worstpossible}
\left|{f_{1}}^{-1}(E)\right|>\sup_{f\in{U}}\left|f^{-1}(E)\right|-\frac{\epsilon}{2}.
\end{equation}
Next choose $\sigma>0$ such that $B(f_{1},\sigma)\subseteq U$
%\begin{equation}\label{sigmaball}
%B(f_{1},\sigma)\subseteq{U} 
%\end{equation}
and set $M=\left\|f_{1}'\right\|+\sigma$. We recall that $B(f_{1},\sigma)$ denotes the open ball in $\Gamma_{1}(H)$ with centre $f_{1}$ and radius $\sigma$ with respect to the $\Gamma_{1}(H)$ norm, defined in Notation~\ref{notation}.

 Fix a positive real number $\lambda>12/\sigma$ large enough so that
\begin{align}\label{lambda}
&(\epsilon\lambda^{2}\left(\sigma/8-1/\lambda\right)^{2}-1)\log(1-Q/\lambda)<-\log{4},\\
&\mbox{ and set }N=\left[\lambda^{2}\left(\frac{\sigma}{8}-\frac{1}{\lambda}\right)^{2}\epsilon\right],\label{eq:Ndef}
\end{align}
where $[-]:\mathbb{R}\to\mathbb{Z}$ denotes the integer-part function.
\begin{algorithm}\label{algorithm}
Let $n\geq{1}$.
\begin{enumerate}
\item Choose $\theta_{n}>0$ small enough so that $\left\|f_{n}-f_{1}\right\|+\theta_{n}<\sigma/4$. Let $\mathcal{V}_{n}$ be the collection of closed intervals consisting of
all those intervals from $\mathcal{V}(f_{n},\lambda,\theta_{n})$ which are
contained in the open set $[0,1]\setminus{\bigcup_{i=0}^{n-1}F_{i}}$. Note that
$\mathcal{V}_{n}$ is a Vitali cover of the set $B_{n}=f_{n}^{-1}(E)\setminus{\bigcup_{i=0}^{n-1}F_{i}}$.
\item Let $\mathcal{I}_{n}=\left\{{I_{k}}^{(n)}\right\}_{k=1}^{\infty}$ be a countable collection of pairwise
disjoint intervals ${I_{k}}^{(n)}$ belonging to $\mathcal{V}_{n}$, such that $\left|B_{n}\setminus\left(\bigcup_{k=1}^{\infty}{I_{k}}^{(n)}\right)\right|=0$
and $\sum_{k=1}^{\infty}\left|{I_{k}}^{(n)}\right|<\left|B_{n}\right|+\frac{\epsilon}{2^{n}}$. Choose an integer $K_{n}\geq{1}$ such that \newline $\sum_{k=K_{n}+1}^{\infty}\left|{I_{k}}^{(n)}\right|<\epsilon/2^{n}$.

\item Let $\widetilde{\phi}_{n}=\psi(f_{n},\lambda,\theta_{n},M,\mathcal{I}_{n},K_{n})$ be the piecewise-linear, continuous curve given by Definiton \ref{piecewisedef}.
\item Let $C_{n}$ be the union of all those intervals $I$ amoungst
${I_{1}}^{(n)},\ldots,{I_{K_{n}}}^{(n)}$ with $\left\|f_{n}'(t)+\widetilde{\phi}_{n}'(t)-f_{1}'(t)\right\|\geq\sigma/4$
for some $t\in{I}$. If $C_{n}\neq\emptyset$, choose an integer $L_{n}$ with
$0\leq{L_{n}}\leq{K_{n}}$ such that, after relabelling, $C_{n}=\bigcup_{k=L_{n}+1}^{K_{n}}{I_{k}}^{(n)}$. If $C_{n}=\emptyset$ then put $L_{n}=K_{n}$.
\item Set $\phi_{n}=\psi(f_{n},\lambda,\theta_{n},M,\mathcal{I}_{n},L_{n})$.
\item For each interval ${I_{k}}^{(n)}$ with $1\leq{k}\leq{L_{n}}$, let ${R_{k}}^{(n)}$ denote the open subinterval given by the conclusion of Lemma \ref{holeestimate}. 
\item Let $g_{n}=g(f_{n},\phi_{n},\epsilon/2^{n})$ be the $C^{1}$ curve given by the conclusion of Lemma~\ref{smoothing} and let $T_{n}\subseteq\bigcup_{k=1}^{L_{n}}{I_{k}}^{(n)}$ be a finite union of closed intervals such that $\left|T_{n}\right|\leq\epsilon/2^{n}$ and $t\in T_{n}$ whenever $g_{n}(t)\neq f_{n}(t)+\phi_{n}(t)$ or $\phi_{n}$ is not differentiable at $t$. Set $S_{n}=S(g_{n},f_{n})$ and $A_{n}=A\left(g_{n},f_{n},C_{n}\cup\bigcup_{i=0}^{n-1}F_{i}\right)$.
\item Set $F_{n}=\mbox{Clos}(S_{n})\cup\mbox{Clos}(A_{n})\cup{T_{n}}\cup{C_{n}}$.
\item Pick $\delta_{n}\in(0,\min(\delta_{n-1}/2,\sigma/(2^{n+3})))$ such that 
\begin{itemize}
\item $f\in{B(f_{1},\sigma)}$ and $\left\|f-f_{1}\right\|<\sigma/4$ for all curves $f\in{B(g_{n},\delta_{n})}$,
\item $\left|f^{-1}(E)\cap{F_{n}}\right|\leq\left|C_{n}\right|+7\epsilon/2^{n}$ whenever \eqref{condition} is satisfied with \newline $(\gamma,g,S,A,\delta)=(f,g_{n},S_{n},A_{n},2\delta_{n})$ 
\end{itemize}
Set $U_{n}=B(g_{n},\delta_{n})$.
\item Choose $f_{n+1}\in{U_{n}}$ satisfying $\left|{f_{n+1}}^{-1}(E)\right|>\sup_{f\in{U_{n}}}\left|f^{-1}(E)\right|-\epsilon/2^{n+1}$.
\end{enumerate}
\end{algorithm}
It is clear that we may choose $\theta_{1}$ as in step 1. For $n> 1$, we can choose $\theta_{n}$ as in step 1 because, by steps 9 and 10, the curve $f_{n}$ satisfies $\left\|f_{n}-f_{1}\right\|<\sigma/4$. Further, steps 9 and 10 guarantee that each $f_{n}$ lies inside $B(f_{1},\sigma)$. In particular, we have $\left\|f_{n}'\right\|\leq\left\|f_{1}'\right\|+\sigma=M$ for all $n$; we require this condition in order to define $\widetilde{\phi_{n}}$ according to Definition~\ref{piecewisedef} at step 3.

The Vitali Covering Theorem allows us to choose closed intervals
${I_{k}}^{(n)}$ as in step 2. We observe that
\begin{equation}\label{totalintervalsum}
\sum_{k=1}^{\infty}\left|{I_{k}}^{(n)}\right|\leq\left|B_{n}\right|+\frac{\epsilon}{2^{n}}=\left|{f_{n}}^{-1}(E)\right|-\left|f_{n}^{-1}(E)\cap\bigcup_{i=0}^{n-1}F_{i}\right|+\frac{\epsilon}{2^{n}}.
\end{equation}
By Lemma \ref{holeestimate} the intervals ${R_{k}}^{(n)}\subset{I_{k}}^{(n)}$, determined at step 6, satisfy 
\begin{equation}\label{Q/lambda}
\left|{R_{k}}^{(n)}\right|\geq\frac{Q}{\lambda}\left|{I_{k}}^{(n)}\right|,
\end{equation}
where $Q=c/4M$. Lemma \ref{smoothing} allows us to choose $T_{n}$ as in step 7. Further, step 7 and Lemma \ref{smoothing} imply $\left\|g_{n}-f_{n}\right\|<\theta_{n}$. Therefore, by our choice of $\theta_{n}$ in step 1 of Algorithm \ref{algorithm} we have that $\left\|g_{n}-f_{1}\right\|<\sigma/4$. For almost all $t\in[0,1]$ (i.e. all $t$ where $\phi_{n}$ is differentiable), we have
\begin{align*}
\left\|g_{n}'(t)-f_{1}'(t)\right\|&\leq\left\|g_{n}'(t)-f_{n}'(t)\right\|+\left\|\phi_{n}'(t)\right\|+\left\|f_{n}'(t)+\phi_{n}'(t)-f_{1}'(t)\right\|\\
&\leq 3\sup\left\{\left\|\phi_{n}'(s)\right\|\mbox{ : }\phi_{n}\mbox{ is differentiable at }s\right\}+\sigma/4\\
&\leq 3/\lambda+\sigma/4<\sigma/2,
\end{align*}
using Lemma~\ref{smoothing}, part 2, Observation~\ref{piecewise}, parts 2--3, Algorithm~\ref{algorithm}, steps 3--5 and $\lambda>12/\sigma$. Hence, $\left\|g_{n}'-f_{1}'\right\|\leq\sigma/2$ and $\left\|g_{n}-f_{1}\right\|<\sigma/4$; we deduce that $g_{n}\in{B(f_{1},\sigma)}\subseteq{U}$.

Note that each of the sets in the union on the right hand side of the equation for $F_{n}$ in step 8 are finite unions of closed intervals. Hence each $F_{n}$ is a finite union of closed intervals. This is important for step 9, as it allows us to apply Lemma~\ref{banachmazur} with the set $F=C_{n}\cup\bigcup_{i=0}^{n-1}F_{i}$.

To pick $\delta_{n}$ as in 9, we use $\left\|g_{n}-f_{1}\right\|<\sigma/4$, $g_{n}\in B(f_{1},\sigma)$ and apply Lemma \ref{banachmazur} with $\theta=\theta_{n}$, $\epsilon,\upsilon=\epsilon/2^{n}$, $U=U_{n-1}$, $f=f_{n}$, $F=C_{n}\cup\bigcup_{i=1}^{n-1}F_{i}$, $\mathcal{I}=\left\{I_{k}\right\}_{k=1}^{\infty}$ where
\begin{equation}
I_{k}=\begin{cases} {I_{k}}^{(n)} & \mbox{ if }k\leq L_{n},\\
I_{K_{n}-L_{n}+k}^{(n)} & \mbox{ if }k>L_{n},\end{cases}
\end{equation}
$L=L_{n}$, $\psi=\phi_{n}$, $R_{k}={R_{k}}^{(n)}$, $g=g_{n}$, $S=S_{n}$ and $A=A_{n}$. We note that condition~\eqref{int2} of Lemma~\ref{banachmazur} is satisfied using Algorithm~\ref{algorithm}, steps~2 and 4.

By construction, the sets $\left\{C_{i}\right\}_{i=1}^{n}$ are pairwise disjoint, and the sets $\left\{F_{i}\right\}_{i=1}^{n}$ are pairwise disjoint up until a set of measure zero. In particular we have
\begin{equation}\label{Cndisjoint}
\left|\bigcup_{i=0}^{n}C_{i}\right|=\sum_{i=0}^{n}\left|C_{i}\right|\mbox{ and }\left|f^{-1}(E)\cap\bigcup_{i=0}^{n}F_{i}\right|=\sum_{i=0}^{n}\left|f^{-1}(E)\cap{F_{i}}\right|
\end{equation}
for all $n\geq{0}$ and all curves $f\in{\Gamma_{1}(H)}$. 

Finally we emphasise that the interval $[0,1]$ can be decomposed as below. This fact is readily verified from Algorithm~\ref{algorithm}, steps 7--8 and Lemma~\ref{banachmazur}. %Lemma~\ref{smoothing} and Lemma~\ref{holeestimate}.
\begin{equation}\label{partition}
[0,1]=\bigcup_{i=1}^{n}F_{i}\cup\bigcup_{k=1}^{L_{n}}\left({I_{k}}^{(n)}\setminus{{R_{k}}^{(n)}}\right)
\end{equation}
\begin{observation}\label{constantobservation}
Suppose $1\leq{p}<n$. Note that the points where $\phi_{p}$ is not
differentiable are contained in the set $F_{n-1}$. This follows from steps
7 and 8 of Algorithm \ref{algorithm}. Hence, the piecewise linear map
$\phi_{p}$ coincides with an affine map when restricted to any connected
component of the set $[0,1]\setminus{F_{n-1}}$. Moreover, by step 1 of
Algorithm \ref{algorithm} we have that the intervals ${I_{k}}^{(n)}$ are
connected subsets of $[0,1]\setminus{F_{n-1}}$. Therefore, we may conclude
that $\phi_{p}$ coincides with an affine map when restricted to each of the
intervals ${I_{k}}^{(n)}$. In particular, $\phi_{p}$ is differentiable on
each interval ${I_{k}}^{(n)}$ and $\phi_{p}'$ is constant when restricted
to each interval ${I_{k}}^{(n)}$.
\end{observation}
We now establish an upper bound for the measure of $E$ on any curve in the open ball $U_{n}$.
\begin{lemma}\label{measure}
For $n\geq 0$ and $f\in U_{n}$, we have
\begin{enumerate}[(i)]
\item $\left|f^{-1}(E)\cap F_{m}\right|\leq\left|C_{m}\right|+7\epsilon/2^{m}$ whenever $0\leq m\leq n$.
\item $\left|f^{-1}(E)\right|\leq\left(1-\frac{Q}{\lambda}\right)^{n}\left|f_{1}^{-1}(E)\right|+\sum_{i=0}^{n}\left|C_{n}\right|+8\sum_{i=0}^{n}\frac{\epsilon}{2^{i}}$.
\end{enumerate}
\end{lemma}
\begin{proof}
For part (i), we note that the case $m=n$ is immediate from Algorithm~\ref{algorithm} step 9. The case $m=0$ is trivial because $F_{0}=C_{0}=\emptyset$. Suppose now that $1\leq m<n$. From Algorithm~\ref{algorithm}, we have that each interval ${I_{k}}^{(j)}$ with $1\leq k\leq L_{j}$ and $j>m$, does not intersect the set $S_{m}\cup A_{m}$. Therefore, using Algorithm~\ref{algorithm}, step~7 and Lemma~\ref{smoothing}, part~4 we have that $(g_{j}(t),g_{j}'(t))=(f_{j}(t),f_{j}'(t))$ for all $t\in S_{m}\cup A_{m}$ and $m+1\leq j\leq n$. We infer that for all $t\in S_{m}\cup A_{m}$,
\begin{multline}
\left\|f(t)-g_{m}(t)\right\|+\left\|f'(t)-g_{m}'(t)\right\|\leq\left\|f(t)-g_{n}(t)\right\|+\left\|f'(t)-g_{n}'(t)\right\|\\
+\sum_{j=m}^{n-1}(\left\|f_{j+1}(t)-g_{j}(t)\right\|+\left\|f_{j+1}'(t)-g_{j}'(t)\right\|)\leq\sum_{i=m}^{n}\delta_{i}\leq 2\delta_{m},
\end{multline}
using $\delta_{i+1}\leq\delta_{i}/2$ for the final inequality. Hence, \eqref{condition} is satisfied for $(\gamma,g,S,A,\delta)=(f,g_{m},S_{m},A_{m},2\delta_{m})$. Part (i) now follows from the conditions imposed by Algorithm~\ref{algorithm}, step 9.

Focusing now on part (ii), we will prove that
\begin{equation}\label{bound2}
\left|f^{-1}(E)\right|\leq\left(1-\frac{Q}{\lambda}\right)^{n}\left|f_{1}^{-1}(E)\right|+\sum_{i=0}^{n}\left|f^{-1}(E)\cap F_{i}\right|+\sum_{i=0}^{n}\frac{\epsilon}{2^{i}}\qquad \forall f\in U_{n}.
\end{equation}
Once this is established, part (ii) of the present Lemma follows from part (i).

For $n=0$, \eqref{bound2} is immediate from \eqref{worstpossible} and
$U_{0}=U$. Suppose $1\leq{j}\leq N$ and that \eqref{bound2}
holds for $n=j-1$. From \eqref{Q/lambda} we have that
\begin{equation*}\label{EonIk}
\left|{I_{k}}^{(j)}\right|-\left|{R_{k}}^{(j)}\right|\leq\left(1-\frac{Q}{\lambda}\right)\left|{I_{k}}^{(j)}\right|\mbox{ for }1\leq{k}\leq{L_{j}}
\end{equation*}
We can now deduce the following sequence of inequalities, using \eqref{totalintervalsum} and applying the induction hypothesis to the curve $f_{j}\in{U_{j-1}}$. 
\begin{align*}
\sum_{k=1}^{L_{j}}\left(\left|{I_{k}}^{(j)}\right|-\left|{R_{k}}^{(j)}\right|\right)&\leq\left(1-\frac{Q}{\lambda}\right)\sum_{k=1}^{L_{j}}\left|{I_{k}}^{(j)}\right| \nonumber \\
&\leq\left(1-\frac{Q}{\lambda}\right)\left(\left|{f_{j}}^{-1}(E)\right|-\sum_{i=0}^{j-1}\left|f_{j}^{-1}(E)\cap F_{i}\right|+\frac{\epsilon}{2^{j}}\right) \nonumber \\
&\leq\left(1-\frac{Q}{\lambda}\right)^{j}\left|f_{1}^{-1}(E)\right|+\sum_{i=0}^{j}\frac{\epsilon}{2^{i}}. \label{I-Rs}
\end{align*}
Let $f$ be a curve in $U_{j}$.  Applying \eqref{partition} and \eqref{Cndisjoint}, we obtain
\begin{align*}
\left|f^{-1}(E)\right|&\leq\sum_{k=1}^{L_{j}}\left(\left|{I_{k}}^{(j)}\right|-\left|{R_{k}}^{(j)}\right|\right)+\sum_{i=0}^{j}\left|f^{-1}(E)\cap{F_{i}}\right|\\
&\leq\left(1-\frac{Q}{\lambda}\right)^{j}\left|f_{1}^{-1}(E)\right|+\sum_{i=0}^{j}\left|{f}^{-1}(E)\cap F_{i}\right|+\sum_{i=0}^{j}\frac{\epsilon}{2^{i}}.
\end{align*}
\end{proof}
\begin{remark}\label{aeremark}
In what follows we will consider the set $[0,1]$ together with the Lebesgue measure and the $\sigma$-algebra of Lebesgue measurable subsets of $[0,1]$ as a probability space. Measurable functions from $[0,1]$ to $H$ will be viewed as random variables on this probability space. In particular, when $\phi$ is a piecewise-linear mapping from $[0,1]$ to $H$, the mapping $t\mapsto\phi'(t)$ defines a random variable on $[0,1]$. Here, we use the identification of $\phi'(t)$ with an element of $H$, given by Notation~\ref{notation}. In such a scenario we will also set
\begin{equation*}
\left\|\phi'\right\|=\sup\left\{\left\|\phi'(t)\right\|:\phi\mbox{ is differentiable at }t\right\}.
\end{equation*}
Finally, for a random variable $X$ on $[0,1]$, the expectation of $X$ is denoted by $\mathbb{E}(X)$.
\end{remark}
\begin{definition}\label{rvs}
For $1\leq{n}\leq{N}$ we define a sequence of random variables $X_{n}$ and a sequence of $\sigma$-algebras $\mathcal{F}_{n}$ on the probability space $[0,1]$ by
\begin{equation*}
X_{n}=\begin{cases} 0 & \mbox{ if }n=1 \\
\sum_{i=1}^{n-1}\phi_{i}' & \mbox{ if }n\geq{1}
\end{cases}\mbox{ and }\mathcal{F}_{n}=\sigma(X_{n}).
\end{equation*}
where the mappings $\phi_{i}$ are given by step~5 of Algorithm~\ref{algorithm}. 
\end{definition}
%The next observation is easily verified using Lemma \ref{martingalelemma}, Observation %\ref{constantobservation} and Definition \ref{rvs}. It will allow us to apply Theorem %\ref{kolmogorov} (Kolmogorov's Martingale Theorem) to the sequence of random variables %$X_{1},\ldots,X_{N}$.
\begin{lemma}\label{lemma:martingale}
The sequence $X_{1},\ldots,X_{N}$ is a martingale on the probability space $[0,1]$, with respect to the filtration $\mathcal{F}_{n}$. Moreover, we have $\mathbb{E}\left(\left\|X_{n}\right\|^{2}\right)<\infty$ for all $n$.
\end{lemma}
\begin{proof}
The $X_{i}$ are trivially $\mathcal{F}_{n}$ measurable and satisfy $\mathbb{E}(\left\|X_{n}\right\|),\mathbb{E}(\left\|X_{n}\right\|^{2})<\infty$. Therefore, we only need to verify $\mathbb{E}(X_{n+1}|A)=\mathbb{E}(X_{n}|A)$ for all $A\in\mathcal{F}_{n}$.

Let $\mathcal{S}_{n}$ be the collection of all subsets $S$ of $[0,1]$ with
the property that for all $k$, either $I_{k}^{(n)}\subseteq{S}$ or
$I_{k}^{(n)}\cap{S}=\emptyset$. Observe that $\mathcal{S}_{n}$ is a $\sigma$-algebra. We claim that
$\mathcal{F}_{n}$ is a sub-$\sigma$-algebra of $\mathcal{S}_{n}$. Assume the claim is valid and pick a set $A\in\mathcal{F}_{n}$. Since $A$ belongs to the $\sigma$-algebra $\mathcal{S}_{n}$ we have that each interval ${I_{k}}^{(n)}$ is either a subset of $A$ or is disjoint from $A$. Let ${I_{k_{1}}}^{(n)},\ldots,{I_{k_{s}}}^{(n)}$ be those intervals which are contained in $A$. By Algorithm \ref{algorithm}, step 5, Definition \ref{piecewisedef} and Observation \ref{piecewise} we have that $\phi_{n}(t)=0$ whenever $t\notin\bigcup_{k=1}^{L_{n}}{I_{k}}^{(n)}$. Further, $\int_{{I_{k}}^{(n)}}\phi_{n}'(t)dt=0$ for all $n$ and $1\leq k\leq L_{n}$. Using these facts, we obtain
\begin{align*}
\mathbb{E}(X_{n+1}|A)&=\mathbb{E}(\phi_{n}'|A)+\mathbb{E}(X_{n}|A)\\
&=\sum_{i=1}^{s}\mathbb{E}(\phi_{n}'|{I_{k_{i}}}^{(n)})+\mathbb{E}(X_{n}|A)=\sum_{i=1}^{s}\int_{{I_{k_{i}}}^{(n)}}\phi_{n}'+\mathbb{E}(X_{n}|A)\\
&=\mathbb{E}(X_{n}|A).
\end{align*}
To complete the proof, we need to verify the earlier claim that $\mathcal{F}_{n}$ is a sub-$\sigma$-algebra of $\mathcal{S}_{n}$. Since each $\phi_{i}'$ is piecewise constant, the $\sigma$-algebra $\mathcal{F}_{n}$ is generated by the collection of sets
\begin{equation*}
\mathcal{J}_{n}=\left\{(\phi_{i}')^{-1}(v)\mbox{ : }v\in{H}\mbox{, }1\leq{i}\leq{n-1}\right\}.
\end{equation*}
It is enough to show that $\mathcal{J}_{n}$ is contained in $\mathcal{S}_{n}$. Let $J\in\mathcal{J}_{n}$ and write $J=(\phi_{i}')^{-1}(v)$ for some $v\in{H}$ and $1\leq{i}\leq{n-1}$. By Observation \ref{constantobservation} we have that ${\phi_{i}}'$ is constant on each of the intervals ${I_{k}}^{(n)}$. Therefore, each interval ${I_{k}}^{(n)}$ is either a subset of $J$ or is disjoint from $J$. Hence $J$ belongs to the $\sigma$-algebra $\mathcal{S}_{n}$.
\end{proof}
\begin{lemma}\label{stopsetlemma}
\begin{equation*}
\left|\bigcup_{i=1}^{N}C_{i}\right|<2\epsilon.
\end{equation*}
\end{lemma}
\begin{proof}
From 4 of Algorithm \ref{algorithm} we have that the set $C_{n}$ is a
finite union of pairwise disjoint closed intervals ${I_{k}}^{(n)}$. Let
${I_{k}}^{(n)}$ be one of the intervals composing $C_{n}$. Then, by Algorithm~\ref{algorithm}, step~4 and Observation~\ref{piecewise}, part~2 we have
\begin{align*}
\frac{\sigma}{4}\leq\sup_{t\in{{I_{k}}^{(n)}}}\left\|f_{n}'(t)+\widetilde{\phi}_{n}'(t)-f_{1}'(t)\right\|&\leq\sup_{t\in{{I_{k}}^{(n)}}}\left\|f_{n}'(t)-f_{1}'(t)\right\|+\left\|\widetilde{\phi}_{n}'\right\|\\
&\leq\sup_{t\in{{I_{k}}^{(n)}}}\left\|f_{n}'(t)-f_{1}'(t)\right\|+\frac{1}{\lambda}.
\end{align*}
The inequality above implies
\begin{equation}\label{Cninequality}
\sup_{t\in{{I_{k}}^{(n)}}}\left\|f_{n}'(t)-f_{1}'(t)\right\|\geq\frac{\sigma}{4}-\frac{1}{\lambda}.
\end{equation}
By steps 9 and 10 of Algorithm \ref{algorithm} we have that
$\left\|f_{i}'(t)-g_{i-1}'(t)\right\|<\delta_{i-1}$ for all
$t\in[0,1]$ and each $i\geq{2}$. Moreover, using step 9 and Definition \ref{rvs}, 
\begin{align}
\left\|f_{n}'(t)-f_{1}'(t)\right\|&\leq\sum_{i=1}^{n-1}\left\|f_{i+1}'(t)-g_{i}'(t)\right\|+\left\|\sum_{i=1}^{n-1}g_{i}'(t)-f_{i}'(t)\right\|\nonumber \\
&\leq\sum_{i=1}^{n-1}\delta_{i}+\left\|\sum_{i=1}^{n-1}\phi_{i}'(t)\right\|+\sum_{i=1}^{n-1}\left\|g_{i}'(t)-(f_{i}'(t)+\phi_{i}'(t))\right\|\nonumber \\
&<\frac{\sigma}{8}+\left\|X_{n}(t)\right\|+\sum_{i=1}^{n-1}\left\|g_{i}'(t)-(f_{i}'(t)+\phi_{i}'(t))\right\|.\label{Cnineq2}
\end{align}
Combining \eqref{Cninequality} and \eqref{Cnineq2} yields
\begin{equation*}
\sup_{t\in{{I_{k}}^{(n)}}}\left(\left\|X_{n}(t)\right\|+\sum_{i=1}^{n-1}\left\|g_{i}'(t)-(f_{i}'(t)+\phi_{i}'(t))\right\|\right)\geq\frac{\sigma}{8}-\frac{1}{\lambda}.
\end{equation*}
By Observation \ref{constantobservation} we have that $\phi_{i}'$ is constant when restricted to each of the intervals
$I_{k}^{(n)}$ whenever $1\leq{i}\leq{n-1}$. Therefore, the map $X_{n}$ is constant when restricted to
each of the intervals ${I_{k}}^{(n)}$. Hence, we may rewrite the
inequality above as
\begin{equation*}
\left\|X_{n}(t)\right\|+\sup_{s\in{{I_{k}}^{(n)}}}\left(\sum_{i=1}^{n-1}\left\|g_{i}'(s)-(f_{i}'(s)+\phi_{i}'(s))\right\|\right)\geq\frac{\sigma}{8}-\frac{1}{\lambda}\mbox{
for all }t\in{{I_{k}}^{(n)}}.
\end{equation*}
In view of Algorithm~\ref{algorithm}, step~7, we may conclude that
\begin{equation*}\label{Cncriterion}
\left\|X_{n}(t)\right\|>\frac{\sigma}{8}-\frac{1}{\lambda}\mbox{
for all }t\in{C_{n}}\setminus\left(\bigcup_{i=1}^{n-1}T_{i}\right),
\end{equation*}
and this subsequently implies
\begin{equation*}
\left|\bigcup_{n=1}^{N}C_{n}\right|\leq\left|\left\{t\in[0,1]\mbox{ :
}\max_{1\leq{n}\leq{N}}\left\|X_{n}(t)\right\|>\frac{\sigma}{8}-\frac{1}{\lambda}\right\}\right|+\epsilon.
\end{equation*}
From Lemma~\ref{lemma:martingale}, the sequence $(X_{n})$ is a martingale with $\mathbb{E}(\left\|X_{n}\right\|^{2})<\infty$. Applying Theorem~\ref{kolmogorov} (Kolmogorov's Martingale Theorem),
we obtain
\begin{equation}\label{martingale1}
\left|\left\{t\in[0,1]\mbox{ :
}\max_{1\leq{n}\leq{N}}\left\|X_{n}(t)\right\|>\frac{\sigma}{8}-\frac{1}{\lambda}\right\}\right|\leq\frac{\mathbb{E}\left(\left\|X_{N}\right\|^{2}\right)}{\left(\frac{\sigma}{8}-\frac{1}{\lambda}\right)^{2}}.
\end{equation}
We now proceed to calculate an estimate for
$\mathbb{E}\left(\left\|X_{N}\right\|^{2}\right)$. From Definition~\ref{rvs}, we have that
\begin{equation*}
\mathbb{E}\left(\left\|X_{N}\right\|^{2}\right)=\sum_{n=1}^{N-1}\mathbb{E}\left(\left\|\phi_{n}'\right\|^{2}\right)+2\sum_{N-1\geq n>m}\mathbb{E}\left(\langle{\phi_{n}',\phi_{m}'}\rangle\right)
\end{equation*}
Recall, using Algorithm \ref{algorithm}, step 5 and Observation \ref{piecewise} that $\left\|\phi_{i}'\right\|<1/\lambda$ for each $i$. We claim that $\mathbb{E}\left(\langle{\phi_{n}',\phi_{p}'}\rangle\right)=0$
whenever $n>p$. Assuming that this claim is correct, the equation above gives
\begin{equation}\label{martingale2}
\mathbb{E}\left(\left\|X_{N}\right\|^{2}\right)=\sum_{n=1}^{N-1}\mathbb{E}\left(\left\|\phi_{n}'\right\|^{2}\right)\leq\sum_{n=1}^{N-1}\left(\frac{1}{\lambda}\right)^{2}<\frac{N}{\lambda^{2}}.
\end{equation}
Now, from \eqref{martingale1}, \eqref{martingale2} and finally \eqref{eq:Ndef}, we have
\begin{equation*}
\left|\left\{t\in[0,1]\mbox{ :
}\max_{1\leq{n}\leq{N}}\left\|X_{n}(t)\right\|>\frac{\sigma}{8}-\frac{1}{\lambda}\right\}\right|<\frac{N}{\lambda^{2}\left(\frac{\sigma}{8}-\frac{1}{\lambda}\right)^{2}}\leq\epsilon.
\end{equation*}
We complete the proof by verifying the earlier claim. Suppose $1\leq{p}<n\leq{N-1}$. Then by Observation \ref{piecewise}, part 1 we have that $\phi_{n}$ is given by
\begin{equation*}\label{phin}
\phi_{n}(t)=\begin{cases} {\varphi_{(k,n)}}{e_{k}}^{(n)} & \mbox{ if
}t\in{{I_{k}}^{(n)}} \\
0 & \mbox{ otherwise,}\end{cases}
\end{equation*}
where the functions $\varphi_{(k,n)}:{I_{k}}^{(n)}\to{H}$ satisfy $\int_{{I_{k}}^{(n)}}{\varphi_{(k,n)}}'=0$, and the ${e_{k}}^{(n)}$ are elements of $H$. Moreover, by Observation
\ref{constantobservation} we have that $\phi_{p}'$ is constant on each of the intervals
${I_{k}}^{(n)}$. We write $\phi_{p}'(t)={v_{k}}^{(n)}$ for all $t\in{{I_{k}}^{(n)}}$, where the ${v_{k}}^{(n)}$ are points in $H$. Now, we have
\begin{equation*}
\mathbb{E}\left(\langle{\phi_{n}',\phi_{p}'}\rangle\right)=\int_{[0,1]}\langle{\phi_{n}',\phi_{p}'}\rangle=\sum_{k=1}^{L_{n}}\langle{{e_{k}}^{(n)},{v_{k}}^{(n)}}\rangle\int_{{I_{k}}^{(n)}}{\varphi_{(k,n)}}'=0.
\end{equation*}
as required.
\end{proof}
\begin{proof}[Proof of Lemma \ref{mainlemma}]
Set $V=U_{N}$ and note that $V\subset U$ using Algorithm~\ref{algorithm}, step 9 and $B(f_{1},\sigma)\subseteq U$. By Lemma \ref{measure}, part (ii) we have that
\begin{equation*}
\sup_{f\in{V}}\left|f^{-1}(E)\right|\leq\left(1-\frac{Q}{\lambda}\right)^{N}\left|f_{1}^{-1}(E)\right|+\left|\bigcup_{i=0}^{N}C_{i}\right|+8\sum_{i=0}^{N}\frac{\epsilon}{2^{i}}.
\end{equation*}
Combining this inequality with Lemma~\ref{stopsetlemma} we obtain
\begin{equation*}
\sup_{f\in{V}}\left|f^{-1}(E)\right|<\left(1-\frac{Q}{\lambda}\right)^{N}\sup_{f\in{U}}\left|f^{-1}(E)\right|+18\epsilon.
\end{equation*}
Observe that \eqref{lambda} and \eqref{eq:Ndef} imply $\left(1-\frac{Q}{\lambda}\right)^{N}<\frac{1}{4}$, whilst $18\epsilon<\frac{1}{4}\sup_{f\in{U}}\left|f^{-1}(E)\right|$ from \eqref{eq:epsilon}.
\end{proof}
We are now ready to prove our main result:
\begin{theorem}\label{mainresult}
Every $\sigma$-porous subset of a Hilbert space $H$ is $\Gamma_{1}$-null in $H$.
\end{theorem}
\begin{proof}
Let $P$ be a $\sigma$-porous subset of $H$ and suppose $c\in(0,1/2)$. By a theorem of Zajicek, \cite[Theorem~4.5]{zajicek75}, we can write $P=\bigcup_{m=1}^{\infty}E_{m}$ where each set $E_{m}$ is a $c$-porous subset of $H$. Fix $m\in\mathbb{N}$ and pick any open ball $U$ in $\Gamma_{1}(H)$. Then, by repeatedly applying Lemma \ref{mainlemma}, we can construct a sequence $\left\{{U_{n}}^{(m)}\right\}_{n=1}^{\infty}$ of open balls in $\Gamma_{1}(H)$ satisfying ${U_{n+1}}^{(m)}\subseteq{{U_{n}}^{(m)}}\subseteq{U}\subset{\Gamma_{1}(H)}$ and
\begin{equation*}
\sup_{f\in{{U_{n}}^{(m)}}}\left|f^{-1}(E_{m})\right|<\frac{1}{2^{n}}\sup_{f\in{U}}\left|f^{-1}(E_{m})\right|<\frac{1}{2^{n}}\mbox{ for all }n\geq 1.
\end{equation*}
Let ${S_{n}}^{(m)}$ be the union of all open balls $V$ in $\Gamma_{1}(H)$ with the property \newline $\sup_{f\in{V}}\left|f^{-1}(E_{m})\right|<\frac{1}{2^{n}}$. By the above discussion we have that each set ${S_{n}}^{(m)}$ has non-empty intersection with any open subset of $H$. Hence ${S_{n}}^{(m)}$ is dense in $\Gamma_{1}(H)$. Clearly, ${S_{n}}^{(m)}$ is also open. Thus, the set $R=\bigcap_{n,m\in\mathbb{N}}{S_{n}}^{(m)}$ is a residual subset of $\Gamma_{1}(H)$. Let $\gamma$ be a curve in $R$. Then,
\begin{equation*}
\left|\gamma^{-1}(E_{m})\right|<\frac{1}{2^{n}}\mbox{ for all }n,m\in\mathbb{N}.
\end{equation*}
This implies $\left|\gamma^{-1}(E_{m})\right|=0$ for all $m\in\mathbb{N}$, and we are done.
\end{proof}
\section{Power-porosity and $\Gamma_{1}$}
In the present section we prove that, unlike $\sigma$-porous sets, power-$p$-porous sets need not be $\Gamma_{1}$-null. Recall that porous and $\sigma$-porous subsets of a Euclidean space have Lebesgue measure zero. We begin by showing that this property does not extend to all power-$p$ porous sets.
\begin{lemma}\label{lemma:powerpcantor}
If $p>1$, there exists a power-$p$-porous subset $C$ of
$[0,1]$ with positive Lebesgue measure.
\end{lemma}
\begin{proof}
We shall construct a Cantor type set $C$ with the desired properties. 

Let $C$ be the Cantor set obtained, following the normal construction. On step $n$, remove $2^{n-1}$ intervals, each of length $\lambda^{n}$ in the usual manner (instead of length $3^{-n}$ for the ternary Cantor set). We can compute the Lebesgue measure of $C$ explicitly: $\left|C\right|=\frac{1-3\lambda}{1-2\lambda}>0$.

We now show that $C$ is a power-$p$-porous set. Fix $x\in{C}$ and $\epsilon>0$. Let $N$ be a natural number large enough so
that
\begin{equation}\label{eq:N}
2^{-(N+1)}<\epsilon\mbox{ and
}N\geq\frac{\log{2}}{\log{(2^{p}\lambda)}}-1.
\end{equation}
Since $x\in{C}$, there exists a remaining interval of the $N$th step
containing $x$. Call this interval ${R}$. By construction, there is a
deleted interval $D$ of the $(N+1)$th step of length $\lambda^{N+1}$ whose
midpoint $h$ coincides with the midpoint of $R$. Let $h$ denote the midpoint of
$D$ and $R$. Then $D=B(h,\lambda^{N+1}/2)$ is a ball which lies inside the
complement of $C$. Moreover,
\begin{equation}\label{eq:h-x}
\left|h-x\right|\leq\mbox{length}(R)/2\leq{2^{-(N+1)}}<\epsilon.
\end{equation}
Together, \eqref{eq:N} and \eqref{eq:h-x} imply that $\frac{\lambda^{N+1}}{2}\geq\left(\frac{1}{2^{N+1}}\right)^{p}\geq\left|h-x\right|^{p}$.
\end{proof}
The next Lemma reveals that subsets of $[0,1]$ with positive Lebesgue measure give rise to a class of
subsets of the plane which have positive Lebesgue measure on all curves
belonging to some open subset of $\Gamma_{1}=\Gamma_{1}(\mathbb{R}^{2})$.

\begin{lemma}\label{lemma:FtimesR}
Suppose $F$ is a subset of $[0,1]$ with positive Lebesgue measure and let
$S=F\times\mathbb{R}$. There exists an open subset $U$ of $\Gamma_{1}$ such that $S$
has positive Lebesgue measure on all curves in $U$.
\end{lemma}
\begin{proof}
Let $\gamma(t)=(t,0)$ be the horizontal interval. It is now easy to check that for sufficiently small $\delta>0$, we have $\left|\rho^{-1}(S)\right|>0$ for all curves $\rho\in B(\gamma,\delta)$.
\end{proof}
From the next Lemma, we see that measure zero subsets of $\mathbb{R}^{2}$ can fail dramatically to be $\Gamma_{1}$-null.
\begin{lemma}\label{lemma:containsresidual}
There is a subset $T$ of $\mathbb{R}^{2}$ which has
measure zero and contains the image of all curves belonging to a residual
subset of $\Gamma_{1}$.
\end{lemma}
\begin{proof}
Let $\left\{\gamma_{n}\right\}_{n=1}^{\infty}$ be a
countable dense subset of $\Gamma_{1}$. For each $\epsilon>0$ and $n\in\mathbb{N}$,
let $U_{n}(\epsilon)$ be an open neighbourhood of $\gamma_{n}([0,1])$ of Lebesgue measure less than $\epsilon/2^{n}$. Consider an open neighbourhood of $\gamma_{n}$ defined by
\begin{equation*}
Y_{n}(\epsilon)=\left\{\gamma\in{\Gamma_{1}}:\gamma(t)\in{U_{n}(\epsilon)}\mbox{ for
all }t\in[0,1]\right\}
\end{equation*}
Letting $Y(\epsilon)=\bigcup_{n=1}^{\infty}Y_{n}(\epsilon)$, we get an open, dense subset of $\Gamma_{1}$ containing the
sequence $\left\{\gamma_{n}\right\}_{n=1}^{\infty}$. Finally, letting $Y=\bigcap_{m=1}^{\infty}Y(\epsilon_{m})$, where $\epsilon_{m}\to 0$, we get a residual subset of $\Gamma_{1}$. Hence, the set
\begin{equation*}
T=\left\{\gamma(t):\gamma\in{Y},t\in[0,1]\right\}
\end{equation*}
has the desired properties.
\end{proof}
We are now ready to prove the main result of this section.
\begin{theorem}
There exists a measure zero, power-$p$-porous subset $A$ of $\mathbb{R}^{2}$ such that the set of all curves $\gamma\in{\Gamma_{1}}$ on which $A$ has Lebesgue measure zero, is not a residual subset of $\Gamma_{1}$.
\end{theorem}
\begin{proof} Let $C$ be the positive measure, power-$p$-porous subset of
$[0,1]$ given by the conclusion of Lemma \ref{lemma:powerpcantor}. Define a set $B\subset\mathbb{R}^{2}$ by $B=C\times\mathbb{R}$.

It is clear that $B$ is a power-$p$-porous subset of $\mathbb{R}^{2}$. Further, we can apply Lemma \ref{lemma:FtimesR} to deduce that there is an open subset $U$ of $\Gamma_{1}$ such that $B$ has positive Lebesgue measure on every curve in $U$.

By Lemma \ref{lemma:containsresidual}, there exists a subset $T$ of $\mathbb{R}^{2}$ with measure zero containing the image $\gamma([0,1])$ of all curves $\gamma$ belonging to a residual subset $R$ of $\Gamma_{1}$. We now define $A=B\cap{T}$.

Since $A\subseteq{B}$ we have that $A$ is a power-$p$-porous subset of
$\mathbb{R}^{2}$. Simultaneously, $A$ has measure zero in
$\mathbb{R}^{2}$ because $A\subseteq{T}$. Therefore, it only remains to
show that the curves in $\Gamma_{1}$, avoiding $A$, do not form a residual subset of
$\Gamma_{1}$.

Suppose $\gamma\in{U\cap{R}}$. Since $T$ contains
the image of $\gamma$, we have
\begin{equation*}\label{eq:gammainvA}
\gamma^{-1}(A)=\gamma^{-1}(B)\cap\gamma^{-1}(T)=\gamma^{-1}(B)\cap[0,1]=\gamma^{-1}(B).
\end{equation*}
Recall that $\gamma\in{U}$ and $B$ has positive measure on all curves
in $U$. Hence, $\gamma^{-1}(A)$ has positive Lebesgue measure and we conclude that $A$ has positive Lebesgue measure on all curves in $U\cap{R}$. Since $U$ is open and $R$ is residual in $\Gamma_{1}$, the complement of $U\cap R$ is not residual.
\end{proof}
\bibliographystyle{plain}
\bibliography{biblio}

\begin{thebibliography}{10}

\bibitem{benyaminilindenstrauss00}
Y.~Benyamini and J.~Lindenstrauss.
\newblock {\em {Geometric Nonlinear Functional Analysis}}, volume~48.
\newblock American Mathematical Society, 2000.

\bibitem{dolvzenko67}
EP~Dol{\v{z}}enko.
\newblock {Grani{\v{c}}nye svojstva proizvolnych funkcij}.
\newblock {\em Izv. Akad. Nauk SSSR Ser. Mat. I}, 3:3--14, 1967.

\bibitem{hunt92}
B.R. Hunt, T.~Sauer, and J.A. Yorke.
\newblock Prevalence: a translation-invariant “almost every” on
  infinite-dimensional spaces.
\newblock {\em Bulletin of the American mathematical society}, 27(2):217--238,
  1992.

\bibitem{lindenstrausstiserpreiss12}
J~Lindenstrauss, D.~Preiss, and J.~Ti\v{s}er.
\newblock {\em {Frechet Differentiability of Lipschitz Functions and Porous
  Sets in Banach Spaces}}.
\newblock Princeton University Press, 2012.

\bibitem{maleva07}
O.~Maleva.
\newblock Unavoidable sigma-porous sets.
\newblock {\em Journal of the London Mathematical Society}, 76(2):467--478,
  2007.

\bibitem{mattila99}
P.~Mattila.
\newblock {\em {Geometry of sets and measures in Euclidean spaces: Fractals and
  Rectifiability}}, volume~44.
\newblock Cambridge University Press, 1999.

\bibitem{preiss90}
D.~Preiss.
\newblock {Differentiability of Lipschitz functions on Banach spaces}.
\newblock {\em Journal of Functional Analysis}, 91(2):312--345, 1990.

\bibitem{preisszajicek01}
D.~Preiss and L.~Zaj{\'\i}{\v{c}}ek.
\newblock {Directional derivatives of Lipschitz functions}.
\newblock {\em Israel Journal of Mathematics}, 125(1):1--27, 2001.

\bibitem{speight12}
G.~Speight.
\newblock Surfaces meeting porous sets in positive measure.
\newblock {\em Israel Journal of Mathematics}, pages 1--26, 2012.

\bibitem{taylor97}
J.~C. Taylor.
\newblock {\em An introduction to measure and probability}.
\newblock New York: Springer, 1997.

\bibitem{zajicek75}
Z.~Zaj\'{i}\v{c}ek.
\newblock {Sets of $\sigma$-porosity and sets of $\sigma$-porosity $(q)$}.
\newblock {\em \v{C}asopis pro p\v{e}stov\'{a}ni matematiky}, 1976.

\end{thebibliography}
\end{document}